\documentclass[a4paper,12pt]{article}
\usepackage{latexsym,amssymb,amsmath,amsthm,amsxtra,xypic}
\usepackage{stmaryrd}
\usepackage{amsfonts}
\usepackage{bbm}
\usepackage{amscd}
\usepackage{mathrsfs}

\def\trr{\triangleright}
\def\trl{\triangleleft}

\newtheorem{proposition}{Proposition}[section]
\newtheorem{lemma}[proposition]{Lemma}

\newtheorem{theorem}[proposition]{Theorem}

\theoremstyle{definition}
\newtheorem{definition}[proposition]{Definition}
\newtheorem{example}[proposition]{Example}

\newcommand{\thlabel}[1]{\label{th:#1}}

\newcommand{\selabel}[1]{\label{se:#1}}

\newcommand{\delabel}[1]{\label{de:#1}}

\newcommand{\eqlabel}[1]{\label{eq:#1}}
\newcommand{\equref}[1]{(\ref{eq:#1})}

\newcommand{\Hom}{{\rm Hom}}

\newcommand{\Cc}{\mathcal{C}}

\def\*C{{}^*\hspace*{-1pt}{\Cc}}
\def\text#1{{\rm {\rm #1}}}

\begin{document}

\title{Unified products for Malcev algebras}

\author{Tao Zhang, Ling Zhang, Ruyi Xie}

\date{}
\maketitle

\noindent

\allowdisplaybreaks

\begin{abstract}
The extending structures and unified products for Malcev algebras are developed.
Some special cases of unified products such as crossed products  and matched pair of Malcev algebras are studied.
It is proved that the extending structures can be classified by some non-abelian cohomology theory.
One dimensional flag extending structures of Malcev algebras are also investigated.
\end{abstract}

\maketitle



\section*{Introduction}
As a generalization of the Lie algebra, Malcev algebras (Mal'tsev algebras, Moufang-Lie algebras) were introduced by A. Maltsev in \cite{Ma} as the tangent algebras to locally smooth Moufang loops.
As an algebraic structure, the concept of Malcev algebras was studied later by by many authors, see for example \cite{Ku,Sa,Ya}.
The representation and cohomology theory of Malcev algebras was established by K. Yamaguti in \cite{Ya}.

On the other hand, extending structures for Lie algebras, associative algebras and Leibniz algebras are studied by Agore and Militaru in \cite{AM1,AM2,AM3,AM4,AM5, AM6}.
The extending structures for left-symmetric algebras, associative and Lie conformal algebras has also been studied by Y. Hong and Y. Su in  \cite{Hong1,Hong2,Hong3}.

In this paper, we study extending structures and unified products for  Malcev algebras.
We follow closely to the theory of unified product and extending structures  which were well  developed by A. L. Agore and G. Militaru in \cite{AM3,AM4,AM6}.
Let $M$ be a Malcev algebra and $E$ a vector space containing $M$ as a subspace.
We will describe and classify all Malcev algebras structures on $E$ such that $M$ is a subalgebra of
$E$. We show that associated to any extending structure of $M$ by a complement space $V$, there is a  unified product on the direct sum space  $E\cong M\oplus V$.
We will show how to classify extending structures for  Malcev algebras by using some non-abelian cohomology and deformation map  theory.

The organization of  this paper is as follows.
In the first section, we review some basic facts about Malcev algebras.
In the  second section,  the definition of  extending structures for Malcev algebras is introduced.
We give the necessary and sufficient conditions for a unified product to  form Malcev algebras.
In the last section, we study some special cases of unified products, which include crossed products of Malcev algebras and a matched pair of Malcevalgebras.
We  also study the flag extending structures. Some concrete examples are given at the end of this section.

Throughout this paper, all Malcev algebras are assumed to be over an algebraically closed field $\mathbb{F}$ of characteristic different from 2 and 3.
The space of linear maps from $V$ to $W$ is denoted by $\Hom(V,W)$.

\section{Preliminaries}
In this section, we recall some basic facts about Malcev algebras.
\begin{definition}
 Let $M$ be a linear space on the field $F$, and there is a binary linear operation $[\,,\,]: M\times M\to M$ satisfying the identity:
\begin{eqnarray}
[ x,y ] &=& - [ y,x ]\\
J\left( x,y,[ x,z ] \right) &=& [J(x,y,z), x]
\end{eqnarray}
for any $x,y,z \in M$,
where $J\left( x,y,z \right) = [ [ x,y ],z] +
[ [ y,z ], x ] + [ [ z,x ],y]$, then $\left( M,[\,,\,] \right)$ is called a Malcev algebra.
\end{definition}

A Malcev algebras is a Lie algebra when the Jacobiator  $J\left( x,y,z \right)=0$.

\begin{theorem}
Equation $(2)$ is equivalent to $(3)$.
\begin{eqnarray}
&&[ {[ {x,z} ],[ {y,w} ]} ] = [{[ {[ {x,y} ],z} ],w} ] + [ {[ {[{y,z} ],w} ], x} ]
+ [ {[ {[ {z,w}], x} ],y} ] + [ {[ {[ {w,x} ],y}],z} ]
\end{eqnarray}
\noindent for any $x,y,z,w \in M$.
\end{theorem}

\begin{proof}
Substitute $J\left( {x,y,z} \right) = [
{[ {x,y} ],z} ] + [ {[ {y,z} ], x} ] +[ {[ {z,x} ],y} ]$ into $(2)$ , we get
\[
[ {[ {x,y} ],[ {x,z} ]} ] = [ {[{[ {x,y} ],z} ], x} ] + [ {[ {[ {y,z}], x} ], x} ] + [ {[ {[ {z,x} ], x}],y} ],
\]
for any $x,y,z,w \in M$.
Thus we have
\[
\begin{array}{l}
 [ {[ {w,y} ],[ {x,z} ]} ] + [ {[{x,y} ],[ {w,z} ]} ] \\
 = [ {[ {[ {w,y} ],z} ], x} ] + [{[ {[ {y,z} ], x} ],w} ] + [ {[ {[{x,y} ],z} ],w} ]
 + [ {[ {[ {y,z} ],w} ], x} ] + [{[ {[ {z,w} ], x} ],y} ],
 \end{array}
\]
\[
\begin{array}{l}
 [ {[ {w,y} ],[ {x,z} ]} ] - [ {[{x,y} ],[ {w,z} ]} ] \\
 = [ {[ {[ {y,w} ], x} ],z} ] + [{[ {[ {w,x} ],z} ],y} ] + [ {[ {[{z,w} ], x} ],y} ]
 + [ {[ {[ {w,x} ],y} ],z} ] + [{[ {[ {x,y} ],z} ],w} ].
 \end{array}
\]
Add the above two formulas to get the formula (3).

Conversely, when $w = x$ in equation (3), it is equation (2).
\end{proof}

\begin{definition}
 Let $\left( {M,\left[\,, \,\right]} \right)$ be a Malcev algebra, a left module of $M$ over a vector space $V$ is a bilinear map $\triangleright :{M}\times {V} \to {V}$ such that the following condition holds:
\begin{eqnarray}
&&[x,z] \triangleright (y \triangleright q) = [[x,y],z]\triangleright q
      - x \triangleright ([y,z] \triangleright q) + z\triangleright (y \triangleright (x \triangleright q))
\end{eqnarray}
\noindent for all $ x,y,z\in M, q \in V$.
\end{definition}

\begin{proposition}\label{prop:01}
Let $M$ be a Malcev algebra and $(V,\triangleright)$  be a left module. Then the direct sum vector space $M \oplus V$ is a Malcev algebra with bracket defined by:
\begin{equation}
[(x, u), (y, v)]= ([x,y] , x \triangleright v - y \triangleright u)
\end{equation}
for all $ x, y \in M, u, v \in V$. This is called the semi-direct product of $M$ and $V$.
\end{proposition}

\begin{theorem} \label{thm:02}
Let $M$ be a Malcev algebra, $V$ be a left $M$-module. Assume there is an anti-symmetric bilinear mapping $\omega: M\times M\to V$. Define the bracket on $M \oplus V$ by:
\begin{eqnarray}
&&[(x, u), (y, v)] = ([x,y] , \,x \triangleright v - y \triangleright u + \omega
(x,y))
\end{eqnarray}
for all $x,y \in M, u,v \in V$.
Then $\left( {M \oplus V,\left[\,, \,\right]}\right)$ is a Malcev algebra if and only if $\omega$ satisfying the following identity:
\[
\begin{array}{l}
 \omega ([x,z],[y,t]) + [t,y] \triangleright \omega (x,z) + [x,z]
\triangleright \omega (y,t) \\
 = \omega ([[x,y],z],t) + \omega ([[y,z],t], x) + \omega ([[z,t], x],y) +
\omega ([[t,x],y],z) \\
 \;\;\; + x \triangleright (t \triangleright \omega (y,z)) - x
\triangleright \omega ([y,z],t) + z \triangleright (y \triangleright \omega
(t,x)) - z \triangleright \omega ([t,x],y) \\
 \;\;\; + t \triangleright (z \triangleright \omega (x,y)) - t
\triangleright \omega ([x,y],z) + y \triangleright (x \triangleright \omega
(z,t)) - y \triangleright \omega ([z,t], x),
 \end{array}
\]
for all $ x,y,z,t \in M$. In this case, $\omega$ is called 2-cocycle on $M$.
\end{theorem}

The proof of the above Proposition \ref{prop:01} and Theorem \ref{thm:02} are by direct computations, so we omit the details.

\begin{definition} \delabel{echivextedn}
Let ${M} $ be a Malcev algebra, $E$  be a Malcev algebra such that
${M} $ is a subalgebra of $E$ and $V$ a complement of
${M} $ in $E$. For a linear map $\varphi: E \to E$ we
consider the diagram:
\begin{eqnarray} \eqlabel{diagrama1}
\xymatrix {& {M}  \ar[r]^{i} \ar[d]_{id} & {E}
\ar[r]^{\pi} \ar[d]^{\varphi} & V \ar[d]^{id}\\
& {M}  \ar[r]^{i} & {E}\ar[r]^{\pi } & V}
\end{eqnarray}
where $\pi : E \to V$ is the canonical projection of $E =
{M}  \oplus V$ on $V$ and $i: {M}  \to E$ is the
inclusion map. We say that $\varphi: E \to E$ \emph{stabilizes}
${M} $ if the left square of the diagram \equref{diagrama1} is
commutative, and $\varphi: E \to E$ \emph{stabilizes}
$V$ if the right square of the diagram \equref{diagrama1} is
commutative.

Let $(E,\cdot)$ and $(E,\cdot')$ be two Malcev algebra structures
on $E$ both containing ${M} $ as a subalgebra. $(E,\cdot)$ and $(E,\cdot')$ are called \emph{equivalent}, and we
denote this by $(E, \cdot) \equiv (E, \cdot')$, if
there exists a Malcev algebra isomorphism $\varphi: (E, \cdot)
\to (E, \cdot')$ which stabilizes ${M} $. Denote by $Extd(E,{M} )$ the set of equivalent classes of ${M} $ through $V$.

 $(E,\cdot)$ and $(E,\cdot')$ are called \emph{cohomologous},
and we denote this by $(E,\cdot) \approx (E, \cdot')$, if there exists a Malcev algebra isomorphism
$\varphi: (E, \cdot) \to (E,\cdot')$ which stabilizes ${M} $ and co-stabilizes $V$.
Denote by $Extd'(E,{M} )$  the set of  cohomologous classes of ${M} $ through $V$.
\end{definition}

\section{Unified products for Malcev algebras}

\begin{definition}\delabel{exdatum}
Let $M$ be a Malcev algebra and $V$ a vector space. An
\textit{extending datum of $M$ through $V$} is a system
$\Omega(M, V) = \bigl(\trl, \, \trr,\, \, \omega, \, [\,,\,]  \bigl)$
consisting of two bilinear maps:
\begin{eqnarray*}
\triangleright:M\times V \to V, \quad\triangleleft :M\times V \to M,
\end{eqnarray*}
and two skew-symmetric bilinear maps:
\begin{eqnarray*}
[,\,] :V\times V \to V,\quad \omega: V\times V \to M.
\end{eqnarray*}
Let $\Omega(M, V) = \bigl(\trl, \, \trr,
\, \, \omega, \, [\,,\,]  \bigl)$ be an extending datum. We denote by $M\, \natural V$ the direct sum vector space $M\oplus V$ together
with the  bracket $[\cdot,\cdot]: (M\oplus  V) \times (M\oplus V) \to M\oplus  V$ defined by:
\begin{eqnarray}
[(x, u), (y, v)] = \Big([x,y] + x\triangleleft v - y \triangleleft u + \omega (u,v), x \triangleright v - y \triangleright u+ [u,v] \Big),
\end{eqnarray}
for all $x,y,z,t \in M, u,v,p, q \in V$.
The object
$M \natural V$ is called the \textit{unified product} of $M$ and $V$ if it is a Malcev algebra with the above bracket.
%In this case the extending datum $\Omega(M, V) =\bigl(\trl, \, \trr, \, \, \omega, \, [\,,\,]  \bigl)$ is called a \textit{extending structure} of $M$ through $V$.
\end{definition}

Then the following theorem provides the set of axioms that need to be
fulfilled by an extending datum $\Omega(M, V)$ such
that $M \natural V$ is a unified product.

\begin{theorem}\thlabel{1}
Let $(M,[\,,\,])$ be a Malcev algebra, $V$ a vector space and
$\Omega(M, V)$ an extending datum of $M$ by
$V$. Then $M \natural V$ is a unified product if and
only if the following compatibility conditions hold for all
$x,y,z,t \in M, u,v,p, q \in V$:
\begin{enumerate}
\item[(U1)]$
 [ {[ {x,z} ],y \trl q} ] + [ {x,z}] \trl \left( {y \trr q} \right) \\
{ = }[ {[ {x,y} ],z} ] \trl q + [{[ {y,z} ] \trl q,x} ] - x \trl \left({[ {y,z} ] \trr q} \right)
 + [ {[ {z \trl q,x} ],y} ] - [ {x\trl \left( {z \trr q} \right),y} ] \\
+ y \left( {x \trl \left( {z \trr q} \right)}\right) - [ {[ {x \trl q,y} ],z} ] + [ {y\trl \left( {x \trr q} \right),z} ] - z\trl \left( {y \trl \left( {x \trr q} \right)}
\right) $,

\item[(U2)]
$
 [ {[ {x,z} ],\omega (v,q)} ] + [ {x,z} ]\trl [ {v,q} ] \\
{ = }[ {x \trl v,z} ] \trl q - \left( {z\trl \left( {x \trr v} \right)} \right) \trl q -\omega \left( {z \trl \left( {x \trr v} \right),q}\right) - [ {\left( {z \trl v} \right) \trl q,x} ]\\
 -[ {\omega \left( {z \trr v,q} \right), x} ]+ x\trl [ {z \trr v,q} ] + x \trl \left( {\left( {z \trl v} \right) \trl q} \right) + [ {z \trl q,x} ] \trl v \\
- \left({x \trl \left( {z \trr q} \right)} \right) \trl v- \omega \left( {x \trr \left( {z \trr q} \right),v}\right)- [ {\left( {x \trl q} \right) \trl v,z}] - [ {\omega \left( {x \trr q,v} \right),z} ]\\
 + z \trl [ {x \trr q,v} ] + z \trl\left( {\left( {x \trl q} \right) \trl v} \right)
 $,

\item[(U3)]
$
 [ {x \trl p,y \trl q} ] + \left( {x\trl p} \right) \trl \left( {y \trr q} \right) -\left( {y \trl q} \right) \trl \left( {x \trr p}
\right) + \omega \left( {x \trr p,y \trr q} \right) \\
{ = }\left( {[ {x,y} ] \trl p} \right)\trl q + \omega \left( {[ {x,y} ] \trr p, q}\right) + [ {\left( {y \trl p} \right) \trl q,x}] + [ {\omega \left( {y \trr p, q} \right), x} ] \\
 - x \trl [ {y \trr p, q} ] - x \trl\left( {\left( {y \trl p} \right) \trr q} \right) +[ {[ {\omega \left( {p, q} \right), x} ],y} ] - [{x \trl [ {p, q} ],y} ] \\
 + y \trl \left( {x \trl [ {p, q} ]} \right) -[ {x \trl q,y} ] \trl p + \left( {y\trl \left( {x \trr q} \right)} \right) \trl p +\omega \left( {y \trl \left( {x \trr q} \right),p}
\right)
$,

\item[(U4)]
$
 [ {x \trl p,\omega (v,q)} ] + \left( {x \trl p} \right) \trl [ {v,q} ] - \omega (v,q) \trl \left( {x \trr p} \right) + \omega \left( {x \trr p,[ {v,q} ]} \right) \\
{ = }\left( {\left( {x \trl v} \right) \trl p} \right) \trl q + \omega \left( {x \trr v,p} \right) \trl q + \omega \left( {[ {x \trr v,p} ], q}\right) + \omega \left( {\left( {x \trl v} \right) \trr p, q} \right) \\
 + [ {\omega \left( {v,p} \right) \trl q,x} ] + [{\omega \left( {[ {v,p} ], q} \right), x} ] - x \trl [ {[ {v,p} ], q} ] - x \trl \left( {\omega \left( {v,p} \right) \trl q} \right) \\
 + [ {\omega \left( {p, q} \right), x} ] \trl v - \left({x \trl [ {p, q} ]} \right) \trl v - \omega\left( {x \trr [ {p, q} ], v} \right) - \left( {\left( {x\trl q} \right) \trl v} \right) \trl p \\
 - \omega \left( {x \trr q,v} \right) \trl p - \omega\left( {[ {x \trr q,v} ],p} \right) - \omega \left({\left( {x \trl q} \right) \trr v,p} \right)
 $,

\item[(U5)]
$
 [ {\omega \left( {u,p} \right),y \trl q} ] + \omega\left( {u,p} \right) \trl \left( {y \trr q} \right) - \left( {y \trl q} \right) \trl [ {u,p} ] +\omega \left( {[ {u,p} ],y \trr q} \right) \\
{ = } - \left( {\left( {y \trl u} \right) \trl p}\right) \trl q - \omega \left( {y \trr u,p} \right)\trl q - \omega \left( {[ {y \trr u,p} ], q}\right) \\
 - \omega \left( {\left( {y \trl u} \right) \trl p, q}\right) + \left( {\left( {y \trl p} \right) \trl q}\right) \trl u + \omega \left( {y \trr p, q} \right)\trl u \\
 + \omega \left( {[ {y \trr p, q} ], u} \right) + \omega\left( {\left( {y \trl p} \right) \trr q, u} \right) +[ {\omega \left( {p, q} \right) \trl u,y} ] \\
 + [ {\omega \left( {[ {p, q} ], u} \right),y} ] - y\trl [ {[ {p, q} ], u} ] - y \trl\left( {\omega \left( {p, q} \right) \trl u} \right) \\
 + [ {\omega\left({q, u} \right), y} ] \trl p - \left({y \trl [ {q, u} ]} \right) \trl p - \omega\left( {y \trr [ {q, u} ], p} \right)
 $,

\item[(U6)]
$
 [ {\omega \left( {u,p} \right),\omega (v,q)} ] + \omega \left({u,p} \right) \trl [ {v,q} ] - \omega (v,q) \trl [ {u,p} ] + \omega \left( {[{u, p}], [ {v, q} ]} \right) \\
{ = }\left( {\omega \left( {u,v} \right) \trl p} \right)\trl q + \omega \left( {[ {u,v} ],p} \right)\trl q + \omega \left( {[ {[ {u,v} ],p} ], q}\right) \\
 + \omega \left( {\omega \left( {u,v} \right) \trr p, q} \right) +\left( {\omega \left( {v,p} \right) \trl q} \right) \trl u+ \omega \left( {[ {v,p} ], q} \right) \trl u \\
 + \omega \left( {[ {[ {v,p} ], q} ], u} \right) +\omega \left( {\omega \left( {v,p} \right) \trr q, u} \right) +\left( {\omega \left( {p, q} \right) \trl u} \right) \trl v\\
 + \omega \left( {[ {p, q} ], u} \right) \trl v + \omega\left( {[ {[ {p, q} ], u} ], v} \right) + \omega \left({\omega \left( {p, q} \right) \trr u,v} \right) \\
 + \left( {\omega \left( {q, u} \right) \trl v} \right)\trl p + \omega \left( {[ {q, u} ], v} \right)\trl p + \omega \left( {[ {[ {q, u} ], v} ],p}\right) \\
 + \omega \left( {\omega \left( {q, u} \right) \trr v,p} \right)
 $,

\item[(U7)]
$
 [ {x,z} ] \trr [ {v,q} ] \\
{ = } - [ {z \trl \left( {x \trr v} \right),q}
] + [ {x \trl v,z} ] \trr q - \left( {z
\trl \left( {x \trr v} \right)} \right) \trr q
 + x \trr [ {z \trr v,q} ] \\
 + x \trr\left( {\left( {z \trl v} \right) \trl q} \right) - [
{x \trr \left( {z \trr q} \right),v} ]
 + [ {z \trl q,x} ] \trr v - \left( {x
\trl \left( {z \trr q} \right)} \right) \trr v\\
+ z \trr [ {x \trr q,v} ]
 + z \trr \left( {\left( {x \trl q} \right) \trl
v} \right) $,

\item[(U8)]
$
  [ {y,t} ] \trr \left( {x \trr p}\right)\\
 =  t \trr \left( {[ {x,y} ]
\trr p} \right) - x \trr \left( {t \trl \left(
{y \trr p} \right)} \right)
+ y \trr \left( {x \trl \left( {t \trr p}
\right)} \right) - [ {[ {t,x} ],y} ] \trr p
$

\item[(U9)]
$
 [ {x \trr p,[ {v,q} ]} ] + \left( {x\trl p} \right) \trr [ {v,q} ] - \omega (v,q)\trr \left( {x \trr p} \right) \\
{ = }[ {[ {x \trr v,p} ], q} ] + [{\left( {x \trl v} \right) \trr p, q} ] + \left({\left( {x \trl v} \right) \trl p} \right) \trr q + \omega \left( {x \trr v,p} \right) \trr q\\
  - x\trr [ {[ {v,p} ], q} ] - x \trr\left( {\omega \left( {v,p} \right) \trr q} \right) - [ {x \trr [ {p, q} ], v} ] + [ {\omega\left( {p, q} \right), x} ] \trr v\\
 - \left( {x \trl [ {p, q} ]} \right) \trr v - [ {[ {x \trr q,v} ],p} ] - [ {\left({x \trl q} \right) \trr v,p} ] - \left( {\left( {x\trl q} \right) \trr v} \right) \trr p \\
 - \omega \left( {x \trr q,v} \right) \trr p $,

\item[(U10)]
$
 [ {x \trr p,y \trr q} ] + \left( {x\trl p} \right) \trr \left( {y \trr q} \right)- \left( {y \trl q} \right) \trr \left( {x \trr p} \right) \\
{ = }[ {[ {x,y} ] \trr p, q} ] + \left( {[ {x,y} ] \trl p} \right) \trr q - x \trr [ {y \trr p, q} ] - x \trr \left( {\left( {y \trl p} \right)\trr q} \right) \\
+ y \trr \left( {x \trr [{p, q} ]} \right) + [ {y \trr \left( {x \trr q}\right),p} ] - [ {x \trl q,y} ] \trr p + \left( {y\trl \left( {x \trr q} \right)} \right) \trr p
$,

\item[(U11)]
$
 [ {[ {u,p} ],[ {v,q} ]} ] + \omega \left({u,p} \right) \trr [ {v,q} ] - \omega (v,q)\trr [ {u,p} ] \\
{ = }[ {[ {[ {u,v} ],p} ], q} ] +[ {\omega \left( {u,v} \right) \trr p, q} ] + \left({\omega \left( {u,v} \right) \trl p} \right) \trr q +\omega \left( {[ {u,v} ],p} \right) \trr q \\
 + [ {[ {[ {v,p} ], q} ], u} ] + [{\omega \left( {v,p} \right) \trr q, u} ] + \left( {\omega\left( {v,p} \right) \trl q} \right) \trr u + \omega\left( {[ {v,p} ], q} \right) \trr u \\
 + [ {[ {[ {p, q} ], u} ], v} ] + [{\omega \left( {p, q} \right) \trr u,v} ] + \left( {\omega
\left( {p, q} \right) \trl u} \right) \trr v + \omega\left( {[ {p, q} ], u} \right) \trr v \\
 + [ {[ {[ {q, u} ], v} ],p} ] + [{\omega \left( {q, u} \right) \trr v,p} ] + \left( {\omega
\left( {q, u} \right) \trl v} \right) \trr p + \omega\left( {[ {q, u} ], v} \right) \trr p $.
\end{enumerate}
\end{theorem}

Given an extending structure $\Omega({M},V)$, then ${M}$ can be seen a Malcev subalgebra of ${M}\natural V$.
On the contrary, we now prove that any Malcev algebra structure on a vector space $E$ containing ${M}$ as a
subalgebra is isomorphic to a unified product.

\begin{theorem}\thlabel{classif}
Let ${M}$ be a Malcev algebra, $E$ a vector space containing
${M}$ as a subspace and $(E, [\,, \, ])$ a Malcev algebra
structure on $E$ such that ${M}$ is a Lie subalgebra. Then there exists a Malcev extending structure
$\Omega({M}, V) = \bigl(\triangleleft, \, \triangleright,
\,\omega, [\,, \, ] \bigl)$ of ${M}$ trough a subspace $V$
of $E$ and an isomorphism of Malcev algebras $(E, [\,, \, ]) \cong
{M} \,\natural \, V$ that stabilizes ${M}$ and
co-stabilizes $V$.
\end{theorem}

\begin{proof} Since $M$ is a subspace and $E$, there exists a projection map $p: E \to{M}$ such that $p(x) = x$, for all $x \in {M}$.
Then $V := \rm{ker}(p)$ is also a subspace of $E$ and a complement of
${M}$ in $E$. We define the extending datum of
${M}$ through $V$ by the following formulas:
\begin{eqnarray*}
\triangleright = \triangleright_p : M \times {V} \to
{V}, \qquad x \triangleright u &:=& [x, \, u] - p \bigl([x, \, u]\bigl)\\
\triangleleft = \triangleleft_p: M \times {V} \to M,
\qquad x \triangleleft u &:=& p \bigl([x, \,u]\bigl)\\
\omega= \omega_p: V\times V \to M, \qquad \omega(u, v) &:=&
p \bigl([u, \, v]\bigl)\\
{[\,, \,]} = {[\,, \,]}_p: V \times V \to V, \qquad [u, v]
&:=& [u, \, v] - p \bigl([u, \, v]\bigl)
\end{eqnarray*}
for any $x , y\in {M}$ and $u$, $v\in V$. First of all, we
observe that the above maps are all well defined bilinear maps: $u
\triangleleft x \in V$ and $\{u, \, v \} \in V$, for all $u$, $v
\in V$ and $x \in {M}$. We shall prove that
$\Omega({M}, V) = \bigl(\triangleleft, \, \triangleright,
\, \omega, \{-, \, -\} \bigl)$ is an extending structure of
${M}$ trough $V$ and
\begin{eqnarray*}
\varphi: {M} \,\natural \, V \to E, \qquad \varphi(x, u)
:= x+u
\end{eqnarray*}
is an isomorphism of Malcev algebras that stabilizes ${M}$
and co-stabilizes $V$. Now $\varphi: {M} \times V \to E$, $\varphi(x, \, u) := x+u$
is a linear isomorphism between the Malcev algebra $E$ and the direct
product of vector spaces ${M}\oplus V$ with the inverse
given by $\varphi^{-1}(v) := \bigl(p(v), \, v - p(v)\bigl)$, for
all $v \in E$. Thus, there exists a unique Malcev algebra structure
on ${M}\oplus V$ such that $\varphi$ is an isomorphism
of Malcev algebras and this unique bracket on ${M}\oplus V$
is given by
$$
[(x, u), \, (y, v)] := \varphi^{-1} \bigl([\varphi(x, u), \,
\varphi(y, v)]\bigl)
$$
for all $x$, $y \in {M}$ and $u$, $v\in V$. The proof is
completely finished if we prove that this bracket coincides with
the one defined by above associated to the system
$\bigl(\triangleleft_p, \, \triangleright_p, \,\omega_p, {[\,, \, ]}_p
\bigl)$. Indeed, for any $x$, $y \in {M}$ and $u$, $v\in
V$ we have:
\begin{eqnarray*}
[(x, u), \, (y, v)] &=& \varphi^{-1} \bigl([\varphi(x, u), \,
\varphi(y, v)]\bigl)
= \varphi^{-1} \bigl([x, \, y] + [x, \, v] + [u, \, y] + [u, \, v]\bigl)\\
&=& \bigl(p([x, \, y]), [x, \, y] - p([x, \, y])\bigl) +
\bigl(p([x, \, v]), [x, \, v] - p([x, \, v])\bigl)\\
&& + \bigl(p([u, \, y]), [u, \, y] - p([u, \, y])\bigl) +
\bigl(p([u, \, v]), [u, \, v] - p([u, \,v ])\bigl)\\
&=& \Bigl(p([x, \, y]) + p([x, \, v]) + p([u, \, y]) +
p([u, \, v]), \ [x, \, y] + [x, \, v]\\
&&+ [u, \, y] + [u, \, v] - p([x, \, y]) - p([x, \, v])
- p([u, \, y]) - p([u, \, v])\Bigl)\\
&=& \Bigl([x, \, y] + x \triangleleft v - y \triangleleft u , \omega (x,y) + x \triangleright v - y \triangleright u+ [u,v]\Bigl)
\end{eqnarray*}
as needed. Moreover, the following diagram is commutative
\begin{eqnarray*}
\xymatrix {& {M} \ar[r]^{i_{{M}}} \ar[d]_{Id} &
{{M} \,\natural \, V} \ar[r]^{q} \ar[d]^{\varphi} & V \ar[d]^{Id}\\
& {M} \ar[r]^{i} & {E}\ar[r]^{\pi } & V}
\end{eqnarray*}
where $\pi : E \to V$ is the projection of $E = {M}\oplus V$
on the vector space $V$ and $q: {{M} \,\natural \, V} \to
V$, $q (x, u) := u$ is the canonical projection. The proof is now
finished.
\end{proof}

\begin{lemma}
Let $\Omega({M}, V) = \bigl(\triangleleft, \,\triangleright, \, \omega, [\cdot,\cdot] \bigl)$ and $\Omega'({M}, V) = \bigl(\triangleleft ', \,
\triangleright ', \, \omega', [\cdot,\cdot]' \bigl)$ be two Malcev extending structures of ${M}$ though $V$,  ${M} \,\natural \, V$ and $ {M} \,\natural \, ' V$
the associated unified products. Then there exists a bijection between the set of all morphisms of Malcev algebras $\psi: {M} \,\natural \, V \to {M} \,\natural \, ' V$
which stabilizes ${M}$ and the set of pairs $(r, s)$, where $r: V \to {M}$, $s: V \to V$ are two linear maps
satisfying the following compatibility conditions for any $x \in {M}$, $u$, $v \in V$:
\begin{enumerate}
\item[(M1)] $s (u) \triangleleft ' x = s(u \triangleleft x)$,
\item[(M2)] $r(u \triangleleft x) = [r(u),\, x] - u
\triangleright x + s(u) \triangleright ' x$,
\item[(M3)] $s([u,\,v])= [s(u), s(v)]' + s(u) \triangleleft ' r(v) - s(v) \triangleleft 'r(u)$,
\item[(M4)] $r([u, v]) = [r(u), \, r(v)] + s(u) \triangleright ' r(v) - s(v) \triangleright ' r(u) + \omega'
\bigl(s(u), s(v)\bigl) - \omega(u, v)$.
\end{enumerate}
Under the above bijection the homomorphism of Malcev algebras $\psi =
\psi_{(r, s)}: {M} \,\natural \, V \to {M}
\,\natural \, ' V$ corresponding to $(r, s)$ is given for any $x
\in {M}$ and $u \in V$ by:
$$
\psi(x, u) = (x + r(u), s(u)).
$$
Moreover, $\psi = \psi_{(r, s)}$ is an isomorphism if and only if
$s: V \to V$ is an isomorphism and $\psi = \psi_{(r, s)}$
co-stabilizes $V$ if and only if $s = id_V$.
\end{lemma}

We denote by $\mathfrak{T}({M},V)$ the set of all extending structures $\Omega(M, V)$.
It is easy to see that  $\equiv$  and $\approx$ are equivalence relations on the set  $\mathfrak{T}(M,V)$.
By the above constructions  and Lemmas  we obtain the following result.

\begin{theorem}
Let ${M}$ be a Malcev algebra, $E$ a  vector space which contains ${M}$ as a  subspace and $V$ a complement
of ${M}$ in $E$. Then, we get:\\
(1)Denote $\mathcal{E}\mathcal{H}^2(V,{M}):=\mathfrak{T}({M},V)/\equiv$. Then, the map
\begin{eqnarray}
\mathcal{E}\mathcal{H}^2(V,{M})\rightarrow Extd(E,{M}),~~~~\overline{\Omega({M},V)}\rightarrow {M}\natural V
\end{eqnarray}
is bijective, where $\overline{\Omega({M},V)}$ is the equivalence class of $\Omega({M},V)$ under $\equiv$.\\
(2) Denote $\mathcal{U}\mathcal{H}^2(V,{M}):=\mathfrak{T}({M},V)/\approx$. Then, the map
\begin{eqnarray}
\mathcal{U}\mathcal{H}^2(V,{M})\rightarrow Extd'(E,{M}),~~~~\overline{\overline{\Omega({M},V)}}\rightarrow {M}\natural V
\end{eqnarray}
is bijective, where $\overline{\overline{\Omega({M},V)}}$ is the equivalence class of $\Omega({M},V)$ under $\approx$.\\
\end{theorem}

\section{Special cases of unified products}\selabel{cazurispeciale}

In this section, we consider the following problem:  what is conditions for unified products when $M$ and $V$ are all subalgebras of $E$? In fact, we obtain
two special cases of unified product: crossed products and matched pairs of Malcev algebras.

\subsection{Crossed products of Malcev algebras}
\begin{definition}
Let $M$ and $V$ be two Malcev algebras with two bilinear  maps $\trl:{M}\times {V} \to {M}$ and $\omega: V\times V \to M$ where $\omega$ is a skew-symmetric map.
We define on the direct sum vector space $M\oplus V$ with the  bracket $[\cdot,\cdot]: (M\oplus  V) \times (M\oplus V) \to M\oplus  V$ by:
\begin{eqnarray}
[(x, u), (y, v)] = \Big([x,y] + x\triangleleft v - y \triangleleft u + \omega (u,v), [u,v] \Big),
\end{eqnarray}
for all $x,y,z,t \in M, u,v,p, q \in V$.
Then $M\oplus V$  is a  Malcev algebra under the above bracket if and only if  the following compatibility conditions hold:
\begin{enumerate}
\item[(CP1)]$
 [ {[ {x,z} ],y \trl q} ] { = }[ {[ {x,y} ],z} ] \trl q + [{[ {y,z} ] \trl q,x} ]  + [ {[ {z \trl q,x} ],y} ] - [ {[ {x \trl q,y} ],z} ]$,

\item[(CP2)]
$
 [ {[ {x,z} ],\omega (v,q)} ] + [ {x,z} ]\trl [ {v,q} ] \\
{ = }[ {x \trl v,z} ] \trl q - [ {\left( {z \trl v} \right) \trl q,x} ]+ x \trl \left( {\left( {z \trl v} \right) \trl q} \right) + [ {z \trl q,x} ] \trl v \\
- [ {\left( {x \trl q} \right) \trl v,z}] + z \trl\left( {\left( {x \trl q} \right) \trl v} \right)
 $,

\item[(CP3)]
$ [ {x \trl p,y \trl q} ] \\
{ = }\left( {[ {x,y} ] \trl p} \right)\trl q  + [ {\left( {y \trl p} \right) \trl q,x}]  +[ {[ {\omega \left( {p, q} \right), x} ],y} ] - [{x \trl [ {p, q} ],y} ] \\
 + y \trl \left( {x \trl [ {p, q} ]} \right) -[ {x \trl q,y} ] \trl p$,

\item[(CP4)]
$
 [ {x \trl p,\omega (v,q)} ] + \left( {x \trl p} \right) \trl [ {v,q} ]\\
{ = }\left( {\left( {x \trl v} \right) \trl p} \right) \trl q - x \trl [ {[ {v,p} ], q} ] - x \trl \left( {\omega \left( {v,p} \right) \trl q} \right) \\
 + [ {\omega \left( {v,p} \right) \trl q,x} ] + [{\omega \left( {[ {v,p} ], q} \right), x} ]  + [ {\omega \left( {p, q} \right), x} ] \trl v \\
 - \left({x \trl [ {p, q} ]} \right) \trl v - \omega\left( {x \trr [ {p, q} ], v} \right) - \left( {\left( {x\trl q} \right) \trl v} \right) \trl p $,

\item[(CP5)]
$
 [ {\omega \left( {u,p} \right),y \trl q} ] - \left( {y \trl q} \right) \trl [ {u,p} ]\\
{ = } - \left( {\left( {y \trl u} \right) \trl p}\right) \trl q  - \omega \left( {\left( {y \trl u} \right) \trl p, q}\right) + \left( {\left( {y \trl p} \right) \trl q}\right) \trl u  \\
 +[ {\omega \left( {p, q} \right) \trl u,y} ] + [ {\omega \left( {[ {p, q} ], u} \right),y} ] - y\trl [ {[ {p, q} ], u} ] - y \trl\left( {\omega \left( {p, q} \right) \trl u} \right) \\
 + [ {\omega\left({q, u} \right), y} ] \trl p - \left({y \trl [ {q, u} ]} \right) \trl p$,

\item[(CP6)]
$
 [ {\omega \left( {u,p} \right),\omega (v,q)} ] + \omega \left({u,p} \right) \trl [ {v,q} ] - \omega (v,q) \trl [ {u,p} ] + \omega \left( {[{u, p}], [ {v, q} ]} \right) \\
{ = }\left( {\omega \left( {u,v} \right) \trl p} \right)\trl q + \omega \left( {[ {u,v} ],p} \right)\trl q + \omega \left( {[ {[ {u,v} ],p} ], q}\right) \\
 +\left( {\omega \left( {v,p} \right) \trl q} \right) \trl u+ \omega \left( {[ {v,p} ], q} \right) \trl u \\
 + \omega \left( {[ {[ {v,p} ], q} ], u} \right) +\left( {\omega \left( {p, q} \right) \trl u} \right) \trl v\\
 + \omega \left( {[ {p, q} ], u} \right) \trl v + \omega\left( {[ {[ {p, q} ], u} ], v} \right) + \omega \left({\omega \left( {p, q} \right) \trr u,v} \right) \\
 + \left( {\omega \left( {q, u} \right) \trl v} \right)\trl p + \omega \left( {[ {q, u} ], v} \right)\trl p + \omega \left( {[ {[ {q, u} ], v} ],p}\right)$.
\end{enumerate}
This will be called the \textit{crossed product} of $M$ and $V$ and we denote it by $M\#_\omega^\trl V$.
\end{definition}

\subsection{Skew-crossed products of Malcev algebras}
\begin{definition}
Let $M$ and $V$ be two Malcev algebras with two bilinear  maps $\trl:{M}\times {V} \to {M}$ and $\omega: V\times V \to M$ where $\omega$ is a skew-symmetric map.
We define on the direct sum vector space $M\oplus V$ with the  bracket $[\cdot,\cdot]: (M\oplus  V) \times (M\oplus V) \to M\oplus  V$ by:
\begin{eqnarray}
[(x, u), (y, v)] = \Big([x,y] + \omega (u,v), x \triangleright v - y \triangleright u+ [u,v] \Big),
\end{eqnarray}
for all $x,y,z,t \in M, u,v,p, q \in V$.
Then $M\oplus V$  is a  Malcev algebra under the above bracket if and only if  the following compatibility conditions hold:
\begin{enumerate}
\item[(SP1)]
$
 [ {[ {x,z} ],\omega (v,q)} ]  { = } -[ {\omega \left( {z \trr v,q} \right), x} ]- \omega \left( {x \trr \left( {z \trr q} \right),v}\right) - [ {\omega \left( {x \trr q,v} \right),z} ]
  $,

\item[(SP2)]
$
\omega \left( {x \trr p,y \trr q} \right) { = } \omega \left( {[ {x,y} ] \trr p, q}\right) + [ {\omega \left( {y \trr p, q} \right), x} ] +[ {[ {\omega \left( {p, q} \right), x} ],y} ]
$,

\item[(SP3)]
$
 \omega \left( {x \trr p,[ {v,q} ]} \right) \\
{ = } \omega \left( {[ {x \trr v,p} ], q}\right) + [{\omega \left( {[ {v,p} ], q} \right), x} ] - \omega\left( {x \trr [ {p, q} ], v} \right)  - \omega\left( {[ {x \trr q,v} ],p} \right)
 $,
\item[(SP5)]
$
 [ {x,z} ] \trr [ {v,q} ]= z \trr [ {x \trr q,v} ] - [{x \trr \left( {z \trr q} \right),v} ]$,

\item[(SP6)]
$ [ {y,t} ] \trr \left( {x \trr p}\right) =  t \trr \left( {[ {x,y} ]\trr p} \right) - [ {[ {t,x} ],y} ] \trr p$,

\item[(SP7)]
$
 [ {x \trr p,[ {v,q} ]} ] - \omega (v,q)\trr \left( {x \trr p} \right) \\
{ = }[ {[ {x \trr v,p} ], q} ] - x\trr [ {[ {v,p} ], q} ]  + \omega \left( {x \trr v,p} \right) \trr q  - x \trr\left( {\omega \left( {v,p} \right) \trr q} \right) \\
- [ {x \trr [ {p, q} ], v} ]  - [ {[ {x \trr q,v} ],p} ]+ [ {\omega\left( {p, q} \right), x} ] \trr v - \omega \left( {x \trr q,v} \right) \trr p $,

\item[(SP8)]
$
 [ {x \trr p,y \trr q} ]  = [ {[ {x,y} ] \trr p, q} ]  - x \trr [ {y \trr p, q} ]  + y \trr \left( {x \trr [{p, q} ]} \right) + [ {y \trr \left( {x \trr q}\right),p} ]$,

 \item[(SP4)]
$
 [ {\omega \left( {u,p} \right),\omega (v,q)} ] + \omega \left( {[{u, p}], [ {v, q} ]} \right) \\
{ = } \omega \left( {[ {[ {u,v} ],p} ], q}\right) + \omega \left( {\omega \left( {u,v} \right) \trr p, q} \right) \\
 + \omega \left( {[ {[ {v,p} ], q} ], u} \right) +\omega \left( {\omega \left( {v,p} \right) \trr q, u} \right) \\
+ \omega\left( {[ {[ {p, q} ], u} ], v} \right) + \omega \left({\omega \left( {p, q} \right) \trr u,v} \right) \\
 + \omega \left( {[ {[ {q, u} ], v} ],p}\right)  + \omega \left( {\omega \left( {q, u} \right) \trr v,p} \right)
 $,

\item[(SP9)]
$ \omega \left({u,p} \right) \trr [ {v,q} ] - \omega (v,q)\trr [ {u,p} ] \\
{ = }[ {\omega \left( {u,v} \right) \trr p, q} ]  +\omega \left( {[ {u,v} ],p} \right) \trr q \\
 + [{\omega \left( {v,p} \right) \trr q, u} ]  + \omega\left( {[ {v,p} ], q} \right) \trr u \\
 + [{\omega \left( {p, q} \right) \trr u,v} ] + \omega\left( {[ {p, q} ], u} \right) \trr v \\
+ [{\omega \left( {q, u} \right) \trr v,p} ]  + \omega\left( {[ {q, u} ], v} \right) \trr p $.
\end{enumerate}
This will be called the \textit{skew crossed product} of $M$ and $V$ and we denote it by $M\#_\omega^\trr V$.
\end{definition}

\subsection{Matched pair for Malcev algebras}

\begin{theorem}
 Let $\left( {M, [\,,\,]} \right),\left({V, [\,,\,]}\right)$ be two Malcev algebras. If there are bilinear maps $\triangleright:M  \times V\to V,\, \triangleleft: V\times M \to V$, define
 bracket on $M \oplus V$ by:
\begin{eqnarray}
[(x, u), (y, v)] = \Big([x,y] + x\triangleleft v - y \triangleleft u, x \triangleright v - y \triangleright u+ [u,v] \Big).
\end{eqnarray}
Then ${M  \oplus V}$ is a Malcev algebra under the above bracket if and only if
the following compatibility conditions hold:
\begin{enumerate}
\item[(MP1)]
\begin{eqnarray*}
\notag&&  [ [x\trl u, y ], z ]- [[y, z ], x ]\trl u- [ y\trl (x\trr u ), z ]\\
\notag&&- [ [z\trl u, x ], y ]+  z\trl (y\trr (x\trr u ) )+ [[x, z ], y\trl u ]\\
\notag&&- [ [x, y ], z ]\trl u+  [x, z ] \trl (y \trr u )+ [x\trl (z\trr u ), y]\\
&&+  x\trl ( [y, z ] \trr u ) - y\trl (x\trr (z\trr u ))=0,
\end{eqnarray*}
\item[(MP2)]
\begin{eqnarray*}
\notag&& [ [x\trr u, v ], w ]- [x\trr [v, w ], u ]- [(x\trl u )\trr  v, w]\\
\notag&&- [ [x\trr w, u ], v ]+ (x\trl u )\trl v )\trr w+ [ [u, w ], x\trr v]\\
\notag&&-x\trr [ [u, v ], w ]+ [ (x\trl w )\trr u, v ]+ (x\trl v )\trr [u, w]\\
&&+(x \trl  [v, w ]   )\trr u-((x\trl w )\trl u )\trr v=0,
\end{eqnarray*}
\item[(MP3)]
\begin{eqnarray*}
\notag&&y\trl (x\trr [u, v ] ) -(y\trl v )\trl (x\trr u )  +[v\trl x , y ]\trl u\\
\notag&&- [(y\trl u )\trl v  , x ]- [x\trl [u, v ], y ]+(x\trl u )\trl (y\trr v)\\
\notag&&-x\trl ( [y\trr u, v ] ) -([x, y ]\trl u )\trl v+ [x\trl u, y\trl v ])\\
&&-(y\trl (x\trr v ))\trl u +x\trl ((y\trl u )\trr  v )=0,
\end{eqnarray*}
\item[(MP4)]
\begin{eqnarray*}
\notag&&([x, y ] \trl u) \trr v-(x\trl u )\trr  (y\trr v ) +x\trr [y\trr u, v]\\
\notag&&- [y\trr (x\trr v ), u ]- [ [x, y ] \trr u, v ]+ (y\trl v )\trr (x\trr u)\\
\notag&&- ( [x\trl v, y ] )\trr u-y\trr (x\trr [u, v ] )+ [x\trr u, y\trr v] \\
&&-x\trr ((y\trl u )\trr  v )+ (y\trl (x\trr v )) \trr u=0,
\end{eqnarray*}
\item[(MP5)]
\begin{eqnarray*}
\notag&&[y\trl v, x ]\trl u -y\trl ( [x\trr v, u ] ) - [(x\trl v )\trl u,\  y]\\
\notag&&+[x\trl u, y ]\trl v - [(y\trl u )\trl v, x ]- x\trl [y\trr u, v]\\
\notag&&- (y\trl (x\trr u )\trl v- (x\trl (y\trr v ))\trl u+[x, y ]\trl [u, v]\\
&&+x\trl((y\trl u )\trr  v) + y\trl((x\trl v)\trr u)=0,
\end{eqnarray*}
\item[(MP6)]
\begin{eqnarray*}
\notag&&x\trr [y\trr v, u ]-  [y\trl u, x ] \trr v- [x\trr (y\trr u ), v ]\\
\notag&&+y\trr [x\trr u, v ]- [y\trr (x\trr v ), u ]-  [x\trl v, y ] \trr u\\
\notag&&-y\trr ((x\trl u )\trr  v )-x\trr ((y\trl v )\trr u )+ [x, y ]\trr [u, v]\\
&&+ ( y\trl (x\trr v )  )\trr u+ (x\trl (y\trr u )  )\trr v=0.
\end{eqnarray*}
\end{enumerate}
This is called a matched pair of   two Malcev algebras $M$ and $V$ if the above conditions are satisfied. 
%This Malcev algebra is called the \emph{bicrossed product}  of $M$ and $V$.  
We will denoted it by  $M \, \bowtie V$.
\end{theorem}

\section{Flag extending structures}
In this section, we study the case when $V$ is a one dimensional vector space. This will be  called flag extending structures.

\begin{definition} Let $M$ be a Malcev algebra. Then $(\lambda,D)$ is a twisted derivation of  $M$  if  there exist  maps $\lambda :M \to
k$ and $D:M \to M$ such that the following conditions hold:
\begin{enumerate}
\item[(T1)] $
 [ {[ {x,z} ],D( y )} ] + \lambda ( y)D( {[ {x,z} ]} ) + \lambda ( {[ {y,z}]} )D( x )
 - D( {[ {[ {y,z} ], x} ]} )\\
  + [{\lambda ( z )D( x ),y} ]- [ {[{D( z ), x} ],y} ] - \lambda ( x )\lambda ( z )D( y ) +[ {[ {D( x ),y} ],z}]\\
  - [ {\lambda( x )D( y ),z} ] - D( {[ {[ {x,y} ],z} ]} )+ \lambda( x )\lambda ( y )D( z ) = 0 $,

\item[(T2)]
$
 \lambda ( {D( y )} )D( x ) - \lambda( x )D( {D( y )} ) + D( {[{D( x ),y} ]} )
 - [ {D( {D( y )} ), x} ] + \lambda ( y)D( {D( x )} ) \\
 - D( {D( {[ {x,y}]} )} ) + [ {D( x ),D( y )} ] - D( {\lambda( x )D( y )} ) = 0 $,

 \item[(T3)]
$
\lambda ( {D( {[ {x,y} ]} )} ) - \lambda ( {D( x )} )\lambda ( y )  - \lambda( {[ {D( x ),y} ]} ) + \lambda ( x )\lambda ( {D( y )} ) = 0 $,

 \item[(T4)]
$
 D( {[ {D( x ),y} ]} ) - D( {\lambda( x )D( y )} ) - [ {D( {D( y)} ), x} ]
 + \lambda ( {D( y )} )D( x ) + D({[ {D( y ), x} ]} ) \\
 - D( {\lambda ( y)D( x )} ) - [ {D( {D( x )} ),y} ] + \lambda ({D( x )} )D( y ) = 0 $,

 \item[(T5)]
$\lambda ( {[ {D( x ),y} ]} ) + \lambda( {[ {D( y ), x} ]} ) = 0$,

\item[(T6)]
$\lambda ( {[x,z]} )\lambda ( y ) = \lambda ({[[x,y],z]} ) - \lambda ( x )\lambda ( {[y,z]})$.
\end{enumerate}
The set of twisted derivations is  denoted by ${\mathcal F} \, (M)$.
\end{definition}

\begin{proposition}
Let $M$ be a Malcev algebra and $V$ a vector space of
dimension $1$ with a basis $\{u\}$. Then there exists a bijection
between the set of extending structures of $M$ through $V$
and ${\mathcal F} \, (M)$.

Under the above bijective correspondence the extending datum $\Omega(M, V)$
corresponding to $(\lambda, \, D) \in {\mathcal F} \, (M)$ is given by:
\begin{eqnarray}
&&x \triangleleft u =  D(x),  \quad x\triangleright u = \lambda (x) u, \\
&&\omega(u, u) = 0, \quad  [u , u] = 0.
\end{eqnarray}
In this case, the unified product $M\natural V$ associated to the extending structure is given by
\begin{equation}
[(x, u), (y, u)] =\Big ([x,y] +  D(x) -  D(y) , (\lambda (x) - \lambda (y)) u\Big).
\end{equation}
\end{proposition}

\begin{theorem}
Let $M$ be an algebra of codimension $1$ in the vector space $V$. Then:
$\operatorname{Extd}(V, M) \cong{\mathcal{A}{{\mathcal{H}}}}^{2}(k, M) \cong \mathcal{F}(M) / \equiv$, where $\equiv$ is the equivalence relation on the set $\mathcal{F}(M)$
defined as follows: $\left(\lambda, D,\right) \equiv$ $\left(\lambda^{\prime}, D^{\prime}\right)$
if and only if $\lambda(x)=\lambda^{\prime}(x)$ and there exists a pair
$(r, s) $, where $r: V \to M$, $s: V \to V$ are two
linear maps, such that:
\begin{eqnarray}
&& D^{\prime}(x) = [r(u) , x] + D(x) + \lambda(x) r(u).
\end{eqnarray}
\end{theorem}

\begin{example}
Let $M$ be a 4-dimensional Malcev algebra with a basis $\left\{ {e_1 ,e_2 ,e_3 ,e_4 } \right\}$, if it is not a Lie algebra, there is only one, relations given as follows:

\[
\left[ {e_1 ,e_2 } \right] = e_2 ,\,
\left[ {e_1 ,e_3 } \right] = e_3 ,\,
\left[ {e_1 ,e_4 } \right] = - e_4 ,\,
\left[ {e_2 ,e_3 } \right] = e_4.
\]
\end{example}
Now we compute the set of twisted derivations as follows.

Denote by
\[
D\left( {{\begin{array}{*{20}c}
 {e_1 }   \\
 {e_2 }   \\
 {e_3 }   \\
 {e_4 }   \\
\end{array} }} \right) = \left( {{\begin{array}{*{20}c}
 {a_{11} }   & {a_{12} }   &
{a_{13} }   & {a_{14} }   \\
 {a_{21} }   & {a_{22} }   &
{a_{23} }   & {a_{24} }   \\
 {a_{31} }   & {a_{32} }   &
{a_{33} }   & {a_{34} }   \\
 {a_{41} }   & {a_{42} }   &
{a_{43} }   & {a_{44} }   \\
\end{array} }} \right)\left( {{\begin{array}{*{20}c}
 {e_1 }   \\
 {e_2 }   \\
 {e_3 }   \\
 {e_4 }   \\
\end{array} }} \right).
\]
Then we have
\[
D\left( {e_1 } \right) = a_{11} e_1 + a_{12} e_2 +a_{13} e_3 +a_{14} e_4 ,\,
D\left( {e_2 } \right) = a_{21} e_1 + a_{22} e_2 +a_{23} e_3 +a_{24} e_4,\,
\]
\[
D\left( {e_3 } \right) = a_{31} e_1 + a_{32} e_2 +a_{33} e_3 +a_{34} e_4 ,\,
D\left( {e_4 } \right) = a_{41} e_1 + a_{42} e_2 +a_{43} e_3 +a_{44} e_4.
\]

From $(T6)$ we get
$$\lambda ( {[e_1, e_3]} )\lambda( e_2 ) = \lambda ({[[e_1, e_2],e_3]} ) - \lambda ( e_1 )\lambda ( {[e_2, e_3]}),$$
thus
$$\lambda ( e_3 )\lambda ( e_2 ) = \lambda ( e_4 ) - \lambda ( e_1 )\lambda ( e_4 ).$$
We will discuss different choice of $\lambda$ in the following cases.

Case $I$:
 Let $$\lambda ( e_1 ) = \lambda ( e_3 ) = \lambda ( e_4 ) = 0, \,\lambda ( e_2 ) = \lambda_2,$$
and  substituting this  into the twisted derivation conditions we can obtain the following result.

From $(T5)$,
$$\lambda ( {[ {D( e_1 ),e_2} ]} ) + \lambda( {[ {D( e_2 ),e_1} ]} ) = 0.$$
we get
$$a_{11}\lambda_2 = 0,$$
thus
$$a_{11} = 0.$$

From $(T3)$,
$$\lambda ( {D( {[ {e_1,e_2} ]} )} ) - \lambda ( {D( e_1 )} )\lambda ( e_2 )
 - \lambda( {[ {D( e_1 ),e_2} ]} ) + \lambda ( e_1 )\lambda ( {D( e_2 )} ) = 0,$$
 we get
$$a_{22}\lambda_2 - a_{12}\lambda_2^{2} - a_{11}\lambda_2= 0,$$
thus
$$a_{22} = \lambda_2 a_{12} + a_{11} = \lambda_2 a_{12}.$$

From $(T1)$,
\begin{eqnarray*}
&&[ {[ {e_1,e_3} ],D( e_2 )} ] + \lambda ( e_2)D( {[ {e_1,e_3} ]} ) + \lambda ( {[ {e_2,e_3}]} )D( e_1 )
 - D( {[ {[ {e_2,e_3} ],e_1} ]} )\\
&&  + [{\lambda ( e_3 )D( e_1 ),e_2} ]- [ {[{D( e_3 ),e_1} ],e_2} ] - \lambda ( e_1 )\lambda ( e_3 )D( e_2 ) +[ {[ {D( e_1 ),e_2} ],e_3}\\
&&  - [ {\lambda( e_1 )D( e_2 ),e_3} ] - D( {[ {[ {e_1,e_2} ],e_3} ]} )+ \lambda( e_1 )\lambda ( e_2 )D( e_3 ) = 0,
\end{eqnarray*}
$$ \lambda ( e_2 )D( e_3 ) + [ [ D( e_1 ),e_2 ],e_3] - D( e_4 ) = 0,$$
we get
$$\lambda_2(a_{31} e_1 + a_{32} e_2 +a_{33} e_3 +a_{34} e_4) + a_{11}e_4= a_{41} e_1 + a_{42} e_2 +a_{43} e_3 +a_{44} e_4,$$
thus we obtain
$$a_{41} = \lambda_2 a_{31},\, a_{42} = \lambda_2 a_{32},\, a_{43} = \lambda_2 a_{33},\, a_{44} = a_{11} + \lambda_2 a_{34} = \lambda_2 a_{34}.$$

From $(T4)$
\begin{eqnarray*}
&&D( {[ {D( e_1 ),e_2} ]} ) + \lambda ( {D( e_2 )} )D( e_1 ) - D( {\lambda ( e_2)D( e_1 )} ) - [ {D( {D( e_1 )} ),e_2} ] + \lambda ({D( e_1 )} )D( e_2 ) = 0,
\end{eqnarray*}
we have
\begin{eqnarray*}
&&a_{11}(a_{21} e_1 + a_{22} e_2 +a_{23} e_3 +a_{24} e_4)
+ \lambda_2 a_{22} (a_{11} e_1 + a_{12} e_2 +a_{13} e_3 +a_{14} e_4)\\
&&- (a_{11}a_{11}+a_{12}a_{21}+a_{13}a_{31}+a_{14}a_{41}) e_2
- \lambda_2(a_{11}(a_{11} e_1 + a_{12} e_2 +a_{13} e_3 +a_{14} e_4)\\
&&+ a_{12} (a_{21} e_1 + a_{22} e_2 +a_{23} e_3 +a_{24} e_4)
+a_{13} (a_{31} e_1 + a_{32} e_2 +a_{33} e_3 +a_{34} e_4)\\
&&+a_{14}(a_{41} e_1 + a_{42} e_2 +a_{43} e_3 +a_{44} e_4))
+ \lambda_2 a_{12} (a_{21} e_1 + a_{22} e_2 +a_{23} e_3 +a_{24} e_4) \\
&&= 0,
\end{eqnarray*}
thus we obtain
\begin{eqnarray*}
\lambda_2a_{13} a_{31} + \lambda_2a_{14}a_{41} &= 0,\\
- a_{12}a_{21}-a_{13}a_{31}-a_{14}a_{41} - \lambda_2a_{13} a_{32} - \lambda_2a_{14} a_{42} + \lambda_2a_{22}a_{12} &= 0,\\
\lambda_2a_{22}a_{13} - \lambda_2a_{13}a_{33} - \lambda_2a_{14}a_{43} &= 0,\\
\lambda_2a_{22}a_{14} - \lambda_2a_{13}a_{34}  - \lambda_2a_{14}a_{44} &= 0.
\end{eqnarray*}
Substituting $a_{22} = \lambda_2 a_{12},$ $a_{41} = \lambda_2 a_{31},\, a_{42} = \lambda_2 a_{32},\, a_{43} = \lambda_2 a_{33},\, a_{44} = \lambda_2 a_{34}$ into the above formula, we get
\begin{eqnarray*}
\lambda_2a_{13} a_{31} + \lambda_2^{2}a_{14} a_{31} = \lambda_2 a_{31}(a_{13} + \lambda_2a_{14}) &=0,\\
- a_{12}a_{21}-a_{13}a_{31}-\lambda_2 a_{14} a_{31} - \lambda_2a_{13} a_{32} - \lambda_2^{2}a_{14} a_{32}  + \lambda_2^{2}a_{12}^{2} &= 0,\\
\lambda_2^{2}a_{12}a_{13} - \lambda_2a_{13}a_{33} - \lambda_2^{2}a_{14} a_{33} = \lambda_2^{2}a_{12}a_{13} - \lambda_2 a_{33}(a_{13} + \lambda_2a_{14}) &= 0,\\
a_{12}\lambda_2^{2}a_{14} - \lambda_2a_{13}a_{34}  - \lambda_2^{2}a_{14} a_{34} &= 0.
\end{eqnarray*}
Let $$a_{12} \neq 0, a_{31} \neq 0, a_{32} \neq 0, a_{33} \neq 0, a_{34} \neq 0,$$
Then we get
$$a_{13} + \lambda_2a_{14} = 0.$$
Therefore
$$a_{13} = 0, a_{14} = 0, a_{21} = \lambda_2^{2}a_{12}.$$

From $(T2)$
$$
\lambda ( {D( e_2 )} )D( e_1 ) + D( {[{D( e_1 ),e_2} ]} )
+ \lambda ( e_2)D( {D( e_1 )} ) - D( {D( {[ {e_1,e_2}]} )} ) + [ {D( e_1 ),D( e_2 )} ] = 0,$$
we have
\begin{eqnarray*}
&&\lambda_2 a_{22} (a_{11} e_1 + a_{12} e_2 +a_{13} e_3 +a_{14} e_4)
+ a_{11}(a_{21} e_1 + a_{22} e_2 +a_{23} e_3 +a_{24} e_4) \\
&&+ \lambda_2(a_{11}(a_{11} e_1 + a_{12} e_2 +a_{13} e_3 +a_{14} e_4)
+ a_{12} (a_{21} e_1 + a_{22} e_2 +a_{23} e_3 +a_{24} e_4)\\
&&+a_{13} (a_{31} e_1 + a_{32} e_2 +a_{33} e_3 +a_{34} e_4)
+a_{14}(a_{41} e_1 + a_{42} e_2 +a_{43} e_3 +a_{44} e_4))\\
&&- a_{21}(a_{11} e_1 + a_{12} e_2 +a_{13} e_3 +a_{14} e_4)
- a_{22} (a_{21} e_1 + a_{22} e_2 +a_{23} e_3 +a_{24} e_4) \\
&&-a_{23} (a_{31} e_1 + a_{32} e_2 +a_{33} e_3 +a_{34} e_4)
-a_{24}(a_{41} e_1 + a_{42} e_2 +a_{43} e_3 +a_{44} e_4)\\
&&+ a_{11}a_{22} e_2+a_{11}a_{23} e_3-a_{11}a_{24} e_4+ a_{12}a_{23} e_4 = 0 ,
\end{eqnarray*}
Thus we obtian
$$
\begin{aligned}
a_{23} a_{31} + \lambda_2 a_{24} a_{31} &= 0,\\
a_{23}a_{32} + \lambda_2 a_{24} a_{32}  &= 0,\\
-a_{23}a_{33} -\lambda_2 a_{24} a_{33} &= 0,\\
\lambda_2a_{12}a_{24} - \lambda_2a_{12}a_{24} -a_{23}a_{34} - \lambda_2 a_{24} a_{34} + a_{12}a_{23} &= 0 ,
\end{aligned}
$$
$$ a_{24}= -\frac{a_{23}}{\lambda_2}, a_{12}a_{23} = 0,$$
Therefore
$$ a_{23}= 0, a_{24}= 0.$$

To sum up, the matrix of $D$ under the base $e_1 ,e_2 ,e_3 ,e_4 $ is give by
\[
\mathop D\nolimits_{1} = \left( {{\begin{array}{*{20}c}
 0   & {a_{12} }   & 0   & 0    \\
 \lambda_2^{2}a_{12}   & \lambda_2 a_{12}   & 0   & 0   \\
 {a_{31} }   & {a_{32} }   & {a_{33} }
  & {a_{34} }   \\
 \lambda_2 a_{31}   & \lambda_2 a_{32}   & \lambda_2 a_{33}   & \lambda_2 a_{34}   \\
\end{array} }} \right).
\]

Case $II$: Let $$\lambda ( e_1 ) = \lambda ( e_2 ) = \lambda ( e_4 ) = 0, \,\lambda ( e_3 ) = \lambda_3.$$

Then substituting this  into the twisted derivation conditions we can obtain the following result.

From $(T5)$, this condition is trivial since both sides are zero.

From $(T3)$,
$$\lambda ( {D( e_2 ]} ) ) = 0,$$
we get
$$a_{23}\lambda_3 = 0,$$
thus
$$a_{23} = 0.$$

From $(T1)$,
$$[ e_3 ,D( e_2 ) ] + \lambda ( e_4 )D( e_1 ) + [{\lambda ( e_3 )D( e_1 ),e_2} ]+[ {[ {D( e_1 ),e_2} ],e_3} - D(e_4) = 0,$$
we have
$$\lambda_3a_{11} e_2 + a_{11} e_4  = a_{41} e_1 + a_{42} e_2 +a_{43} e_3 +a_{44} e_4,$$
Thus we obtain
$$a_{41} = 0,\, a_{42} = \lambda_3a_{11},\, a_{43} = 0,\, a_{44} = a_{11}.$$

From $(T4)$,
\begin{eqnarray*}
&&D( {[ {D( e_1 ),e_2} ]} ) + \lambda ( {D( e_2 )} )D( e_1 ) - [ {D( {D( e_1 )} ),e_2} ] + \lambda ({D( e_1 )} )D( e_2 ) = 0,
\end{eqnarray*}
we get
\begin{eqnarray*}
&&a_{11}(a_{21} e_1 + a_{22} e_2 + a_{23} e_3 + a_{24} e_4)
+ \lambda_3 a_{23} (a_{11} e_1 + a_{12} e_2 + a_{13} e_3 + a_{14} e_4)\\
&&- (a_{11}a_{11}+a_{12}a_{21}+a_{13}a_{31}+a_{14}a_{41}) e_2
+ \lambda_3 a_{13} (a_{21} e_1 + a_{22} e_2 +a_{23} e_3 +a_{24} e_4) \\
&&= 0,
\end{eqnarray*}
Thus we obtain
\begin{eqnarray*}
(a_{11} + \lambda_3 a_{13})a_{21} &= 0,\\
(a_{11}+ \lambda_3 a_{13})a_{22} - a_{11}^{2} - a_{12}a_{21} - a_{13}a_{31} &= 0,\\
(a_{11} + \lambda_3 a_{13})a_{24} &= 0,
\end{eqnarray*}

From $(T2)$,
$$
\lambda ( {D( e_2 )} )D( e_1 ) + D( {[{D( e_1 ),e_2} ]} )
 - D( {D( {[ {e_1,e_2}]} )} ) + [ {D( e_1 ),D( e_2 )} ] = 0,$$
we have
\begin{eqnarray*}
&&\lambda_3 a_{23} (a_{11} e_1 + a_{12} e_2 +a_{13} e_3 +a_{14} e_4)
+ a_{11}(a_{21} e_1 + a_{22} e_2 +a_{23} e_3 +a_{24} e_4) \\
&&- a_{21}(a_{11} e_1 + a_{12} e_2 +a_{13} e_3 +a_{14} e_4)
- a_{22} (a_{21} e_1 + a_{22} e_2 +a_{23} e_3 +a_{24} e_4) \\
&&-a_{24}(a_{41} e_1 + a_{42} e_2 +a_{43} e_3 +a_{44} e_4)
+ a_{11}a_{22} e_2+a_{11}a_{23} e_3-a_{11}a_{24} e_4 \\
&&+ a_{12}a_{23} e_4 = 0 ,
\end{eqnarray*}
Thus we obtain
$$
\begin{aligned}
a_{22} a_{21} &= 0,\\
2a_{11}a_{22} - a_{21} a_{12} - a_{22} a_{22} - a_{24} a_{42}  &= 0,\\
a_{21}a_{13}  &= 0,\\
a_{21}a_{14} + a_{22} a_{24} +a_{24} a_{44} &= 0.
\end{aligned}
$$

Let $a_{11} \neq 0,$
then we have $$a_{13} = -\frac{a_{11}}{\lambda_3}, a_{21} = a_{24} = 0, a_{22} = 2a_{11}, a_{31} = \lambda_3 a_{11},$$
or $$a_{13} = -\frac{a_{11}}{\lambda_3}, a_{21} = 0, a_{22} = -a_{11}, a_{31} = -\frac{3a_{11}}{\lambda_3},$$
or $$a_{21} = a_{24} = 0, a_{22} = 0, a_{31} = -\frac{a_{11}^{2}}{a_{13}},$$
or $$a_{21} = a_{24} = 0, a_{22} = 2a_{11}, a_{31} = \frac{(a_{11} + 2\lambda_3 a_{13})a_{11}}{a_{13}}.$$

To sum up, the matrix of $D$ under the base $e_1 ,e_2 ,e_3 ,e_4 $ is given by
\[
\mathop D\nolimits_{21} = \left( {{\begin{array}{*{20}c}
 a_{11}   & {a_{12} }   & -\frac{a_{11}}{\lambda_3}   & a_{14}    \\
 0   & 2a_{11}   & 0   & 0   \\
 {a_{31} }   & {a_{32} }   & {a_{33} }
  & {a_{34} }   \\
 0   & \lambda_3 a_{11}   & 0   & a_{11}   \\
\end{array} }} \right),
\]
\[
\mathop D\nolimits_{22} = \left( {{\begin{array}{*{20}c}
a_{11}   & {a_{12} }   & -\frac{a_{11}}{\lambda_3}   & a_{14}    \\
 0   & -a_{11}   & 0   & a_{24}   \\
 -\frac{3a_{11}}{\lambda_3}   & {a_{32} }   & {a_{33} }
  & {a_{34} }   \\
 0   & \lambda_3 a_{11}   & 0   & a_{11}   \\
\end{array} }} \right),
\]
\[
\mathop D\nolimits_{23} = \left( {{\begin{array}{*{20}c}
 a_{11}   & {a_{12} }   & a_{13}   & a_{14}    \\
 0   & 0   & 0   & 0   \\
 -\frac{a_{11}^{2}}{a_{13}}   & {a_{32} }   & {a_{33} }
  & {a_{34} }   \\
 0   & \lambda_3 a_{11}   & 0   & a_{11}   \\
\end{array} }} \right),
\]
\[
\mathop D\nolimits_{24} = \left( {{\begin{array}{*{20}c}
 a_{11}   & {a_{12} }   & a_{13}   & a_{14}    \\
 0   & 2a_{11}   & 0   & 0   \\
 \frac{(a_{11} + 2\lambda_3 a_{13})a_{11}}{a_{13}}   & {a_{32} }   & {a_{33} }
  & {a_{34} }   \\
 0   & \lambda_3 a_{11}   & 0   & a_{11}   \\
\end{array} }} \right).
\]

Case $III$: Let $$\lambda ( e_1 ) = \lambda ( e_2 ) = \lambda ( e_3 ) = 0, \,\lambda ( e_4 ) = \lambda_4 \neq 0,$$
then substituting this into the twisted derivation condition we can obtain the following result.

From $(T3)$,  we obtian
$$a_{24} = 0.$$

From $(T1)$,
$$ \lambda ( e_4 )D( e_1 ) + [ [ D( e_1 ),e_2 ],e_3] - D( e_4 ) = 0,$$
we have
$$\lambda_4(a_{11} e_1 + a_{12} e_2 +a_{13} e_3 +a_{14} e_4) + a_{11}e_4= a_{41} e_1 + a_{42} e_2 +a_{43} e_3 +a_{44} e_4,$$
Thus we obtain
$$a_{41} = \lambda_4 a_{11},\, a_{42} = \lambda_4 a_{12},\, a_{43} = \lambda_4 a_{13},\, a_{44} = a_{11} + \lambda_4 a_{14}.$$

From $(T4)$,
\begin{eqnarray*}
&&D( {[ {D( e_1 ),e_2} ]} ) + \lambda ( {D( e_2 )} )D( e_1 ) - [ {D( {D( e_1 )} ),e_2} ] + \lambda ({D( e_1 )} )D( e_2 ) = 0,
\end{eqnarray*}
we have
\begin{eqnarray*}
&&a_{11}(a_{21} e_1 + a_{22} e_2 +a_{23} e_3 +a_{24} e_4)
+ \lambda_4 a_{24} (a_{11} e_1 + a_{12} e_2 +a_{13} e_3 +a_{14} e_4)\\
&&- (a_{11}a_{11}+a_{12}a_{21}+a_{13}a_{31}+a_{14}a_{41}) e_2
+ \lambda_4 a_{14} (a_{21} e_1 + a_{22} e_2 +a_{23} e_3 +a_{24} e_4) \\
&&= 0,
\end{eqnarray*}
thus we obtain
\begin{eqnarray*}
(a_{11}+ \lambda_4 a_{14})a_{21} &= 0,\\
(a_{11}+ \lambda_4 a_{14})a_{22} - a_{11}^{2} - a_{12}a_{21} - a_{13}a_{31} - a_{14}\lambda_4 a_{11} &= 0,\\
(a_{11}+ \lambda_4 a_{14})a_{23} &= 0.
\end{eqnarray*}

Let
$$a_{11} \neq 0, a_{12} \neq 0, a_{13} \neq 0, a_{14} \neq 0,$$
Then we obtain
$$a_{21} = a_{23} = 0, (a_{11}+ \lambda_4 a_{14})a_{22} - a_{11}^{2} - a_{13}a_{31} - a_{14}\lambda_4 a_{11} = 0.$$

From $(T2)$,
$$
\lambda ( {D( e_2 )} )D( e_1 ) + D( {[{D( e_1 ),e_2} ]} ) - D( {D( {[ {e_1,e_2}]} )} ) + [ {D( e_1 ),D( e_2 )} ] = 0,$$
we have
\begin{eqnarray*}
&&\lambda_4 a_{24} (a_{11} e_1 + a_{12} e_2 +a_{13} e_3 +a_{14} e_4)
+ a_{11}(a_{21} e_1 + a_{22} e_2 +a_{23} e_3 +a_{24} e_4) \\
&&- a_{21}(a_{11} e_1 + a_{12} e_2 +a_{13} e_3 +a_{14} e_4)
- a_{22} (a_{21} e_1 + a_{22} e_2 +a_{23} e_3 +a_{24} e_4) \\
&&-a_{23} (a_{31} e_1 + a_{32} e_2 +a_{33} e_3 +a_{34} e_4)
-a_{24}(a_{41} e_1 + a_{42} e_2 +a_{43} e_3 +a_{44} e_4)\\
&&+ a_{11}a_{22} e_2+a_{11}a_{23} e_3-a_{11}a_{24} e_4+ a_{12}a_{23} e_4 = 0 ,
\end{eqnarray*}
thus we obtain
$$
\begin{aligned}
a_{22} a_{21} + a_{23} a_{31} &= 0,\\
2a_{11}a_{22} - a_{21}a_{12} - a_{22}^{2} -a_{23} a_{32} &= 0,\\
2a_{11}a_{23} - a_{21}a_{13} - a_{22}a_{23} -a_{23}a_{33}  &= 0,\\
- a_{21}a_{14} - a_{23}a_{34} + a_{12}a_{23} &= 0.
\end{aligned}
$$
Therefore we the following
$$a_{22} = 0, a_{31} = -\frac{a_{11}^{2} + \lambda_4 a_{14} a_{11}}{a_{13}},$$
or $$a_{22} = 2a_{11}, a_{31} = \frac{a_{11}^{2} + \lambda_4 a_{14} a_{11}}{a_{13}}.$$

To sum up, the matrix of $D$ under the base $e_1 ,e_2 ,e_3 ,e_4 $ is
\[
\mathop D\nolimits_{31} = \left( {{\begin{array}{*{20}c}
 {a_{11} }   & {a_{12} }   & {a_{13} }   & {a_{14} }    \\
 0   & 0   & 0   & 0   \\
 -\frac{a_{11}^{2} + \lambda_4 a_{14} a_{11}}{a_{13}}   & {a_{32} }   & {a_{33} }
  & {a_{34} }   \\
 \lambda_4 a_{11}   & \lambda_4 a_{12}   & \lambda_4 a_{13}   & \lambda_4 a_{14} + a_{11}   \\
\end{array} }} \right)
\]
or
\[
\mathop D\nolimits_{32} = \left( {{\begin{array}{*{20}c}
 {a_{11}}   & {a_{12} }   & {a_{13} }   & {a_{14} }    \\
 0   & 2a_{11}   & 0   & 0   \\
 \frac{a_{11}^{2} + \lambda_4 a_{14} a_{11}}{a_{13}}   & {a_{32} }   & {a_{33} }
  & {a_{34} }   \\
 \lambda_4 a_{11}   & \lambda_4 a_{12}   & \lambda_4 a_{13}   & \lambda_4 a_{14} + a_{11}   \\
\end{array} }} \right).
\]

\section*{Acknowledgments}
This is a primary edition. Something should be modified in the future.

\vskip7pt

\footnotesize{
  \noindent % Addresses:
 College of Mathematics, Henan Normal University, Xinxiang 453007, P. R. China;\\
E-mail address:\texttt{{
 zhangtao@htu.edu.cn}}\vskip5pt

\noindent % Addresses:
 College of Mathematics, Henan Normal University, Xinxiang 453007, P. R. China;\\
 E-mail address:\texttt{{
zhanglingny@163.com}}.\vskip5pt

 \footnotesize{\noindent % Addresses:
 College of Mathematics, Henan Normal University, Xinxiang 453007, P. R. China;\\
 E-mail address:\texttt{{
  xieruyi98@163.com}}.\vskip5pt
}


\begin{thebibliography}{99}


\bibitem{AM1}
A. L. Agore, G. Militaru,  \emph{Extending structures I: the level of groups}, {Algebr. Represent. Theory} {17} (2014), 831--848.

\bibitem{AM2}
A. L. Agore, G. Militaru, \emph{Extending structures II: the quantum version},  J. Algebra {336} (2011), 321--341.

%\bibitem{pierce}
%A.L. Agore, G. Militaru, \emph{Unified products and split extensions of Hopf algebras}, Contemporary Math. AMS 585 (2013), 1--15.

\bibitem{AM3}
A.L. Agore, G. Militaru, \emph{Extending structures for Lie algebras}, Monatsh. fur Mathematik 174(2014), 169--193.

\bibitem{AM4}
A. L. Agore, G. Militaru, \emph{Unified products for Leibniz algebras. Applications}, Linear Algebra Appl. {439} (2013), 2609--2633.


\bibitem{AM5}
A.L. Agore, G. Militaru, \emph{The global extension problem, crossed uroducts and co-flag noncommutative Poisson algebras}, J. Algebra 426(2015), 1--31.

\bibitem{AM6}
A. L. Agore, G. Militaru, \emph{Extending structures, Galois groups and supersolvable associative algebras}, Monatsh. Math. {181} (2016), 1--33.


 \bibitem{Hong1}
Y. Hong, Extending structures and classifying complements for left-symmetric algebras, Results Math.74(2019), 32. arXiv:1511.08571.

\bibitem{Hong2}
Y. Hong, Extending structures for associative conformal algebras, Linear Multilinear Algebra 67(2019), 196--212. arXiv:1705.02827.

\bibitem{Hong3}
Y. Hong and Y. Su, Extending structures for Lie conformal algebras, Algebr. Represent. Theor. {20} (2017), 209-230


\bibitem{Fi}
V.T. Filippov, \emph{Mal'tsev algebras}, Algebra and Logic 16(1)(1977), 70--74.

\bibitem{EM}
A. Elduque and H. C. Myung, \emph{Mutations of alternative algebras}, Kluwer Academic Publishers, Boston, 1994.



\bibitem{GM}
M. Gunaydin and D. Minic. \emph{Nonassociativity, Malcev algebras and string theory}, Fortschr. Phys. { 61}(10)(2013), 873--892.



\bibitem{Ku}
E.N. Kuz'min, \emph{Mal'tsev algebras and their representations}, Algebra and Logic { 7} (4) (1968), 233--244.

\bibitem{Ma}
 A. Malcev, \emph{Analytic loops}, Mat. Sb. { 78}(1955), 569--578.



\bibitem{Sa}
A. A. Sagle, \emph{Malcev algebras}, Trans. Amer. Math. Soc. { 101}(1961), 426--458.

\bibitem{Ya0}
K. Yamaguti, \emph{Note on Malcev algebras}, Kumamoto J. Sci., Ser. A { 5} (1962), 203--207.

\bibitem{Ya}
K. Yamaguti, \emph{On the theory of Malcev algebras}, Kumamoto J. Sci. Ser. A { 6} (1963), 9--45.

\bibitem{Ya1}
K. Yamaguti, \emph{On the cohomology space of Lie triple systems}, Kumamoto J. Sci. A { 5} (1960), 44--52.



\end{thebibliography}
\end{document}